\documentclass[11pt]{article}%
\usepackage{amsmath}
\usepackage{amsfonts}
\usepackage{amssymb}
\usepackage{graphicx}%
\providecommand{\U}[1]{\protect\rule{.1in}{.1in}}
\newtheorem{theorem}{Theorem}

\newtheorem{definition}[theorem]{Definition}

\newtheorem{lemma}[theorem]{Lemma}

\newtheorem{proposition}[theorem]{Proposition}
\newtheorem{remark}[theorem]{Remark}

\newenvironment{proof}[1][Proof]{\noindent\textbf{#1.} }{\ \rule{0.5em}{0.5em}}
\begin{document}

\title{On the critical points of the energy functional on vector fields of a
Riemannian manifold}
\author{Giovanni Nunes
\and Jaime Ripoll}
\date{February 11, 2017}
\maketitle

\begin{abstract}
Given a compact Lie subgroup $G$ of the isometry group of a compact Riemannian
manifold $M$ with a Riemannian connection $\nabla,$ it is introduced a
$G-$symmetrization process of a vector field of $M$ and it is proved that the
critical points of the energy functional
\[
F(X):=\frac{\int_{M}\left\Vert \nabla X\right\Vert ^{2}dM}{\int_{M}\left\Vert
X\right\Vert ^{2}dM}%
\]
on the space of $\ G-$invariant vector fields are critical points of $F$ on
the space of all vector fields of $M,$ and that this inclusion may be strict
in general. One proves that the infimum of $F$ on $\mathbb{S}^{3}$ is not
assumed by a $\mathbb{S}^{3}-$invariant vector field. It is proved that the
infimum of $F$ on a sphere $\mathbb{S}^{n},$ $n\geq2,$ of radius $1/k,$ is
$k^{2},$ and is assumed by a vector field invariant by the isotropy subgroup
of the isometry group of $\mathbb{S}^{n}$ at any given point of $\mathbb{S}%
^{n}.$ It is proved that if $G$ is a compact Lie subgroup of the isometry
group of a compact rank $1$ symmetric space $M$ which leaves pointwise fixed a
totally geodesic submanifold of dimension bigger than or equal to $1$ then all
the critical points of $F$ are assumed by a $G-$invariant vector field.

Finally, it is obtained a characterization of the spheres by proving that on a
certain class of Riemannian compact manifolds $M$ that contains rotationally
symmetric manifolds and rank $1$ symmetric spaces$,$ with positive Ricci
curvature $\operatorname*{Ric}\nolimits_{M}$, $F$ has the lower bound
$\operatorname*{Ric}\nolimits_{M}/\left(  n-1\right)  $ among the $G-$
invariant vector fields, where $G$ is the isotropy subgroup of the isometry
group of $M$ at a point of $M,$ and that his lower bound is attained if and
only if $M$ is a sphere of radius $1/\sqrt{\operatorname*{Ric}\nolimits_{M}}.$

\end{abstract}

\section{Introduction}

\qquad Let $M$ be a compact, orientable, $n-$dimensional, $n\geq2,$
$C^{\infty}$ manifold with a Riemannian metric and let $\nabla$ be the
Riemannian connection of $M.$ In this paper we study the critical points of
the energy of $\nabla$ acting on the space $C^{\infty}(TM)$ of $C^{\infty}$
vector fields of $M$ with unit $L^{2}$ norm. Precisely, we study the critical
points of the functional
\[
F(X)=\int_{M}\left\Vert \nabla X\right\Vert ^{2}dM
\]
on the space of vector fields $X\in C^{\infty}(TM)$ such that%
\[
\int_{M}\left\Vert X\right\Vert ^{2}dM=1.
\]

It is well known that the critical values of $F$ are the eingenvalues of the
so called rough Laplacian $-\operatorname*{div}\nabla$ of $M$ \cite{U} and it
follows from the spectral theory for linear elliptic operators that they
constitute an increasing sequence $0\leq\delta_{1}<\delta_{2}<\cdots
\rightarrow+\infty$ (counted with multiplicity) which are assumed by
$C^{\infty}$ eigenvector fields$.$ Moreover, if $M$ has no parallel vector
fields then the infimum $\delta_{1}$ of $F$ is positive.

The search of geometric estimates of the spectrum of\ elliptic linear
operators is an active topic of investigation in Geometric Analysis, the
Laplacian being already a classical and well studied one \cite{R}. As to the
study of the rough Laplacian operator, it seems to the authors that no
attention has been paid so far. We have not found in the literature a
description of its eigenvalues even in the simplest Riemannian spaces as the spheres.

In this paper we study the critical points of $F$ on Riemannian compact
manifolds admitting a nontrivial isometry group. A simple but fundamental idea
here is to introduce a process of symmetrization of a vector field of $M$ by a
compact Lie subgroup $G$ of the isometry group $\operatorname*{Iso}(M)$ of
$M$. We then use this symmetrization to prove that, apart some exceptional
cases where the symmetrization process leads to zero vector fields, the
critical points of $F$ on the space of $G-$invariant vector fields are in the
spectrum of $F$ (Proposition \ref{css1}). What is somewhat surprising is that
this inclusion may be strict, even for the infimum: We prove that the infimum
of $F$ in the unit sphere $\mathbb{S}^{3}$ is not realized by a $\mathbb{S}%
^{3}-$ invariant vector field, considering $\mathbb{S}^{3}$ as a Lie subgroup
of $O(4)$ (Remark \ref{by}); in other words, left invariant vector fields of
$\mathbb{S}^{3}$ (with a bi-invariant metric) are critical points of $F$ but
are not ones with least energy. This raises the problem of whether the infimum
or more generally the spectrum of $F$ is assumed by $G-$invariant vector
fields for a given compact Lie subgroup $G$ of the isometry group of $M$.

We prove that the orthogonal projection $V$ of a nonzero vector $v$ of
$\mathbb{R}^{n+1}$ on $T\mathbb{S}^{n}(1/k),$ where $\mathbb{S}^{n}(1/k)$ is a
sphere of radius $1/k,$ realizes the infimum of $F$ on $C^{\infty}\left(
T\mathbb{S}^{n}(1/k)\right)  $ and that this infimum is $k^{2}$ (Theorem
\ref{inf}) (we note that $V$ is invariant by the isotropy subgroup of $O(n+1)$
that leaves fixed the point $v/\left(  k\left\vert v\right\vert \right)
\in\mathbb{S}^{n}(1/k)).$ Moreover, we prove that if $X$ is a Killing field of
$\mathbb{S}^{n}(1/k)$ then $F(X)=(n-1)k^{2}.$ It then follows that $X$
minimizes the energy if and only if $n=2.$ As a consequence, if $\mathcal{S}%
_{n}\subset C^{\infty}(T\mathbb{S}^{n}(1/k))$ denotes the space of eigenvector
fields associated to $k^{2}$ then $\dim\mathcal{S}_{2}\geq\dim O(3)+\dim
\mathbb{R}^{3}=6.$ By using Fourier series it is proved in \cite{BR} that
actually $\dim\mathcal{S}_{2}=6$.

We also prove that if $M$ is a compact rank $1$ symmetric space and if $G$ is
a Lie subgroup of $\operatorname*{Iso}\left(  M\right)  $ which leaves
pointwise fixed a totally geodesic submanifold of dimension bigger than or
equal to $1$ then all the critical points of $F$ are realized by a
$G-$invariant vector fields (Theorem \ref{main}). We note that there are many
subgroups $G$ of $\operatorname*{Iso}(M)$ satisfying this condition. For
example, any totally geodesic $m-$dimensional submanifold of $\mathbb{S}%
^{n}(1/k)$, $1\leq m\leq n-2,$ is the fixed point of a compact subgroup of
$O(n+1)$ isomorphic to $O(n-m).$ More generally, any rank $1$ symmetric space
is plenty of totally geodesic submanifolds all of them being the fixed points
of some compact Lie subgroup of the isometry group of the space (see
\cite{He}). We also mention the groups of reflections on totally geodesic
submanifolds, which are finite groups isomorphic to $\mathbb{Z}_{2}.$

In the last part of the paper we obtain a characterization of the sphere as a
space where $F$ attains the infimum of the energy on a certain class of
Riemannian manifolds as explained next. We shall say that $M$ is two point
symmetric with center $p\in M$ if the isotropy subgroup $\operatorname*{Iso}%
\nolimits_{p}(M)$ of $\operatorname*{Iso}(M)$ at $p$ is isomorphic to the
isotropy subgroup of the isometry group of a rank $1$ symmetric space, that
is, $\operatorname*{Iso}\nolimits_{p}(M)$ is isomorphic to some of the
following Lie groups: $O(n),$ $U(1)\times U(n-1),$ $Sp(1)\times Sp(n-1)$ or
$\operatorname*{Spin}(9)$ (see \cite{He} and Definition \ref{def} ahead)$.$
Note that when $\operatorname*{Iso}\nolimits_{p}(M)=O(n)$ then $M$ is a
rotationally symmetric space (see \cite{Cho})$.$ Symmetric spaces of rank $1$
satisfy the so called two point homogeneous property and hence are also known
as two point homogeneous spaces.

We prove that if $M$ is a two point symmetric space with center $p$ and with
Ricci curvature $\operatorname*{Ric}\nolimits_{M}$ satisfying
$\operatorname*{Ric}\nolimits_{M}\geq(n-1)k^{2},$ then the infimum of $F$ on
the space of $\operatorname*{Iso}\nolimits_{p}(M)-$invariant vector fields is
bigger than or equal to $k^{2}$ and the equality holds if and only if $M$ is a
sphere of radius $1/k$ (Theorem \ref{ricci}).

\section{\label{sp}A general result}

\qquad Let $M$ be a compact Riemannian manifold of dimension $n\geq2$. We
choose on the full isometry group $\operatorname*{Iso}\left(  M\right)  $ of
$M$ a fixed left invariant Riemannian metric. We consider on any Lie subgroup
of $\operatorname*{Iso}\left(  M\right)  $ the left invariant Riemannian
metric induced by the one of $\operatorname*{Iso}\left(  M\right)  .$

Let $G$ be a compact Lie subgroup of $\operatorname*{Iso}\left(  M\right)  $.
Given a vector field $V\in C^{\infty}(TM)$, the $G-$symmetrization of $V$ is
the vector field $V_{G}$ defined by setting, at a given $p\in M,$%
\[
\left\langle V_{G}(p),u\right\rangle =\frac{1}{\operatorname*{Vol}(G)}\int
_{G}\left\langle (dg_{p})^{-1}V(g(p)),u\right\rangle dg
\]
where $u\in T_{p}M.$ We note that $G$ may be finite case in which $V_{G}$ is
given by
\[
V_{G}(p)=\frac{1}{\left\vert G\right\vert }\sum_{g\in G}dg_{p}^{-1}(V(g(p)),
\]
where $\left\vert G\right\vert $ denotes the number of elements of $G.$ Note
that $V_{G}\in C^{\infty}(TM)$ is $G-$invariant, that is, $V_{G}%
(g(p))=dg_{p}\left(  V_{G}(p)\right)  $ for all $p\in M$ and $g\in G.$
Moreover, by the linearity of $\operatorname{div}\nabla$ and of the
integration process we have%
\[
\left(  \operatorname{div}\nabla V\right)  _{G}=\operatorname{div}\nabla
V_{G}.
\]
In particular, if $V$ satisfies $\operatorname{div}\nabla V=-\lambda V$ then
$\operatorname{div}\nabla V_{G}=-\lambda V_{G}.$ We call $V_{G}$ the $G-$mean
of $V.$

We observe that depending on $M$ and $G$ it may happen that $V_{G}\equiv0.$
This happens with any vector field $V$ on a rank $1$ compact symmetric space
$M$ if $G$ is the full isometry group of isometries of $M$ (this is
consequence of Lemma \ref{trans}). We prove:

\begin{proposition}
\label{css1} Let $M$ be a compact $n-$dimensional Riemannian manifold and $G$
a compact Lie subgroup of $\operatorname*{Iso}\left(  M\right)  .$ Then the
eigenvalues and eigenvectors of $F$ restrict to the subspace of $G-$ invariant
vector fields of $C^{\infty}(TM)$ are also eigenvalues and eigenvectors of $F$
on $C^{\infty}(TM).$
\end{proposition}

We need the following result:

\begin{lemma}
\label{product}On the hypothesis of the Proposition \ref{css1} assume that
$W\in C^{\infty}(TM)$ satisfies $W_{G}\equiv0$. Let $V$ be a $G-$invariant
vector field$.$ Then
\[
\int_{M}\left\langle W(x),V(x)\right\rangle dx=0.
\]

\end{lemma}

\begin{proof}
Since the elements of $G$ are isometries and $V$ is $G-$invariant we have, for
all $g\in G$
\begin{align*}
\int_{M}\left\langle W\left(  x\right)  ,V\left(  x\right)  \right\rangle dx
&  =\int_{M}\left\langle W\left(  g(x)\right)  ,V\left(  g(x)\right)
\right\rangle dx\\
&  =\int_{M}\left\langle dg_{x}^{-1}\left(  W\left(  g(x)\right)  \right)
,dg_{x}^{-1}\left(  V\left(  g(x)\right)  \right)  \right\rangle dx\\
&  =\int_{M}\left\langle dg_{x}^{-1}\left(  W\left(  g(x)\right)  \right)
,V(x)\right\rangle dx.
\end{align*}
It follows from Fubini's theorem that%
\begin{align*}
\int_{M}\left\langle W\left(  x\right)  ,V\left(  x\right)  \right\rangle dx
&  =\frac{1}{\operatorname{Vol}\left(  G\right)  }\int_{G}\int_{M}\left\langle
dg_{x}^{-1}\left(  W\left(  g(x)\right)  \right)  ,V(x)\right\rangle dxdg\\
&  =\int_{M}\left\langle \frac{1}{\operatorname{Vol}\left(  G\right)  }%
\int_{G}dg_{x}^{-1}\left(  W\left(  g(x)\right)  \right)  dg,V(x)\right\rangle
dx\\
&  =\int_{M}\left\langle W_{G}(x),V(x)\right\rangle dx=0.
\end{align*}
proving the lemma.
\end{proof}

\begin{proof}
[Proof of the Proposition \ref{css1}]Let $X$ an eingenvector of
$\operatorname{div}\nabla$ on the space of the $G$-invariant vector fields
associated to the eigenvalue $\lambda.$ Then%
\[
\int_{M}\left\langle -\operatorname{div}\nabla X,W\right\rangle dx=\lambda
\int_{M}\left\langle X,W\right\rangle dx,
\]
for all $G$-invariant vector field $W$. For proving that
\[
\int_{M}\left\langle -\operatorname{div}\nabla X,V\right\rangle dx=\lambda
\int_{M}\left\langle X,V\right\rangle dx,
\]
holds for any given vector field $V\in C^{\infty}(TM)$ we write $V=Z+V_{G}$
with $Z=V-V_{G}$ so that%
\[
\int_{M}\left\langle -\operatorname{div}\nabla X,V\right\rangle dx=\int
_{M}\left\langle -\operatorname{div}\nabla X,Z\right\rangle dx+\int
_{M}\left\langle -\operatorname{div}\nabla X,V_{G}\right\rangle dx.
\]

Noting that $Z$ has zero $G-$mean and and that $\operatorname{div}\nabla X$ is
a $G-$invariant because $X$ is we have, by Lemma \ref{product}, that the first
term of the hand side of the equality above is zero and then%
\begin{align*}
\int_{M}\left\langle -\operatorname{div}\nabla X,V\right\rangle dx  &
=\int_{M}\left\langle -\operatorname{div}\nabla X,V_{G}\right\rangle
dx=\lambda\int_{M}\left\langle X,V_{G}\right\rangle dx\\
&  =\lambda\int_{M}\left\langle X,V-Z\right\rangle dx=\lambda\int
_{M}\left\langle X,V\right\rangle dx
\end{align*}
proving the proposition.
\end{proof}

\begin{remark}
\label{by}Considering a bi-invariant metric on the unit sphere $\mathbb{S}%
^{3}$ with the Lie group structure, $\mathbb{S}^{3}$ is a Lie subgroup of the
isometry group $O(4)$ of $\mathbb{S}^{3}.$ It follows from Proposition
\ref{css1} that the $\mathbb{S}^{3}-$invariant vector fields are eigenvectors
of the rough Laplacian of $\mathbb{S}^{3}.$ Clearly the $\mathbb{S}^{3}%
-$invariant vector fields are the left (and right) invariant vector fields of
$\mathbb{S}^{3}.$ The orbits of a left invariant vector field constitute a
Hopf fibration of $\mathbb{S}^{3}.$

We have that the energy of a left invariant vector field is $2.$\ Indeed, the
Bochner-Yano formula for a vector field $X\in C^{\infty}(T\mathbb{S}^{3})$,
namely,%
\begin{align*}
\int_{\mathbb{S}^{3}}\left\vert \nabla X\right\vert ^{2}dx  &  =\int
_{\mathbb{S}^{3}}\left(  \operatorname*{Ric}(X,X)+2\left\vert
\operatorname*{Kill}(X)\right\vert ^{2}-\left(  \operatorname{div}X\right)
^{2}\right)  dx\\
&  =\int_{\mathbb{S}^{3}}\left(  2\left\vert X\right\vert ^{2}+2\left\vert
\operatorname*{Kill}(X)\right\vert ^{2}-\left(  \operatorname{div}X\right)
^{2}\right)  dx,
\end{align*}
where $\operatorname*{Kill}(X)$ is the $(0,2)-$tensor
\[
\operatorname*{Kill}(X)(U,V)=\frac{\left\langle \nabla_{U}X,V\right\rangle
+\left\langle \nabla_{V}X,U\right\rangle }{2},\text{ }U,V\in C^{\infty
}(T\mathbb{S}^{3})
\]
(see \cite{Y}), when $X$ is a Killing field, since $\operatorname*{Kill}%
(X)\equiv0$ and $\operatorname{div}X\equiv0,$ gives%
\[
F(X)=\frac{\int_{\mathbb{S}^{3}}\left\vert \nabla X\right\vert ^{2}dx}%
{\int_{\mathbb{S}^{3}}\left\vert X\right\vert ^{2}dx}=2.
\]

As we will see in the next section, the infimum of $F$ on $\mathbb{S}^{n}$ is
$1,$ for all $n.$ Thus, the $\mathbb{S}^{3}-$invariant vector fields do not
realize the infimum of energy on $\mathbb{S}^{3}.$
\end{remark}

\section{\label{S}The infimum of the energy on a sphere}

\qquad In this section we determine the infimum of $F$ in the case that $M$ is
a sphere. Denoting by $\mathbb{S}^{n}\left(  1/k\right)  $ the sphere of
radius $1/k$ in $\mathbb{R}^{n+1}$ we prove:

\begin{theorem}
\label{inf}The infimum of $F$ on $C^{\infty}(T\mathbb{S}^{n}\left(
1/k\right)  ),$ $n\geq2,$ is $k^{2}$, and is assumed by the orthogonal
projection on $T\mathbb{S}^{n}\left(  1/k\right)  $ of a constant nonzero
vector field of $\mathbb{R}^{n+1}$. Moreover, a Killing field of
$\mathbb{S}^{n}\left(  1/k\right)  $ realizes the infimum $k^{2}$ of energy if
and only if $n=2.$
\end{theorem}

We need some preliminary lemmas.

\begin{lemma}
\label{zero}Let $X$ be a vector field of $\mathbb{S}^{n}\left(  1/k\right)  $
with zero $G-$mean ($X_{G}\equiv0)$, where $G$ is the isotropy subgroup of
$\operatorname*{Iso}(\mathbb{S}^{n}\left(  1/k\right)  )$ that leaves fixed a
point $v\in$ $\mathbb{S}^{n}\left(  1/k\right)  .$ Then the function
$f(p):=\left\langle X(p),v\right\rangle ,$ $p\in\mathbb{S}^{n}\left(
1/k\right)  ,$ has zero mean in $\mathbb{S}^{n}\left(  1/k\right)  $ that is,
\[
\int_{\mathbb{S}^{n}\left(  1/k\right)  }f(x)dx=0.
\]

\end{lemma}

\begin{proof}
Using the formula of coarea to integrate $f$ on $\mathbb{S}^{n}\left(
1/k\right)  $ along the level sets of $h:\mathbb{S}^{n}\left(  1/k\right)
\rightarrow\mathbb{R},$ $h(p)=d(p,v),$ where $d$ is the distance in
$\mathbb{S}^{n}\left(  1/k\right)  ,$ since $\left\Vert \operatorname{grad}%
h\right\Vert =1,$ we have
\begin{equation}%
{\displaystyle\int_{\mathbb{S}^{n}\left(  1/k\right)  }}
f(x)=\int_{0}^{\pi}\left(  \int_{h^{-1}(t)}f(x)\right)  dt. \label{int}%
\end{equation}

We note that $h^{-1}(t)$ is a geodesic sphere $\mathbb{S}_{t}^{n-1}$ of
$\mathbb{S}^{n}\left(  1/k\right)  $ centered at $v$ and with radius $t$ and
that these spheres are the orbits of $G.$ Given $t\in\left(  0,\pi/k\right)
,$ choose $p\in h^{-1}(t),$ and denote by $H$ the subgroup of isotropy of $G$
at $p$. Let
\[
\psi:\frac{G}{H}\rightarrow\mathbb{S}_{t}^{n-1}%
\]
be given by $\psi(gH)=g(p).$ We note that, up to a factor multiplying the
metric of $G,$ $\psi$ is an isometry. Setting
\[
\widetilde{f}:=f\circ\psi:\frac{G}{H}\rightarrow\mathbb{R}%
\]
we then obtain
\begin{equation}
\int_{h^{-1}(t)}f(x)dx=\int_{\frac{G}{H}}\widetilde{f}(gH)d(gH). \label{til}%
\end{equation}

Let
\[
\phi:G\rightarrow\frac{G}{H},
\]
be the projection of $G$ over $G/H,$ and set $\overline{f}:=\widetilde{f}%
\circ\phi.$ Using the coarea formula to integrate $\overline{f}$ on $G$ along
the fibers of $\phi,$ we obtain
\begin{align*}
\int_{g\in G}\overline{f}(g)dg  &  =\int_{\frac{G}{H}}\left(  \int_{gH}%
\frac{1}{||\operatorname*{Jac}\phi||}\overline{f}(gh)d(gh)\right)  d(gH)\\
&  =\int_{\frac{G}{H}}\left(  \int_{gH}\overline{f}(gh)d(gh)\right)  d(gH).
\end{align*}

Since $\overline{f}$ is constant on $gH,$ $g\in G,$ we get
\[
\int_{G}\overline{f}(g)dg=\operatorname*{Vol}(H)\int_{\frac{G}{H}}%
\widetilde{f}(gH)d(gH)
\]
or by (\ref{til})%
\begin{equation}
\int_{h^{-1}(t)}f(x)dx=\frac{1}{\operatorname*{Vol}(H)}\int_{G}\overline
{f}(g)dg. \label{ff}%
\end{equation}

But
\begin{align*}
\int_{G}\overline{f}(g)dg  &  =\int_{G}\widetilde{f}\left(  gH\right)  dg\\
&  =\int_{G}\left(  f\circ\psi\right)  \left(  gH\right)  dg=\int_{G}f\left(
g(p)\right)  dg\\
&  =\int_{G}\left\langle X(g(p)),v\right\rangle dg=\int_{G}\left\langle
dg_{p}^{-1}X(g(p)),dg_{p}^{-1}v\right\rangle dg\\
&  =\int_{G}\left\langle dg_{p}^{-1}X(g(p)),v\right\rangle dg=\left\langle
X_{G}(p),v\right\rangle =0.
\end{align*}
Therefore from (\ref{ff})
\[
\int_{h^{-1}(t)}f=0,
\]
for any $t\in\left(  0,\pi/k\right)  ,$ and this, with (\ref{int}), proves the lemma.
\end{proof}

\begin{lemma}
\label{laplafunc}Let $V\in C^{\infty}(T\mathbb{S}^{n}\left(  1/k\right)  )$,
$n\geq2,$ with%
\[
\int_{\mathbb{S}^{n}\left(  1/k\right)  }\left\vert V\right\vert ^{2}=1
\]
be given. Let $G$ be the isotropy subgroup of $\operatorname*{Iso}\left(
\mathbb{S}^{n}\left(  1/k\right)  \right)  $ at some point of $\mathbb{S}%
^{n}\left(  1/k\right)  .$ If $V$ is nonzero and has zero $G-$mean, then
$F(V)\geq\left(  n-1\right)  k^{2}.$
\end{lemma}

\begin{proof}
Denote by $\nabla$ the Riemannian connection of $\mathbb{S}^{n}\left(
1/k\right)  .$ Write
\[
V=\overset{n+1}{\underset{l=1}{\sum}}a_{l}e_{l},
\]
where $\{e_{l}\}$ is an orthonormal basis of the $\mathbb{R}^{n+1}$. Fix
$p\in\mathbb{S}^{n}\left(  1/k\right)  \ $and let $\{E_{j}\}$ be an
orthonormal frame of $\mathbb{S}^{n}\left(  1/k\right)  $ on a neighborhood of
$p.$ Then, for each $i$ we have
\[
\nabla_{E_{i}}V=\overset{n+1}{\underset{l=1}{\sum}}E_{i}(a_{l})e_{l}%
-\left\langle V,E_{i}\right\rangle kp
\]
and thus%
\begin{align*}
\left\langle \nabla V,\nabla V\right\rangle  &  =%
{\displaystyle\sum\limits_{i=1}^{n}}
\left\langle \nabla_{E_{i}}V,\nabla_{E_{i}}V\right\rangle \\
&  =\sum_{i=1}^{n}\left(  -\left\langle V,E_{i}\right\rangle ^{2}k^{2}%
+\sum_{l=1}^{n+1}\left\langle E_{i}(a_{l})e_{l},E_{i}(a_{l})e_{l}\right\rangle
\right) \\
&  =\overset{n+1}{-k^{2}\left\vert V\right\vert ^{2}+%
{\displaystyle\sum\limits_{l=1}^{n+1}}
}\left\vert \operatorname{grad}(a_{l})\right\vert ^{2}.
\end{align*}
We then have
\[
F(V)=-k^{2}+\sum_{l=1}^{n+1}\int_{\mathbb{S}^{n}\left(  1/k\right)
}\left\vert \operatorname{grad}(a_{l})\right\vert ^{2}.
\]

By Lemma \ref{zero} the functions $a_{l}$ have zero mean on $\mathbb{S}%
^{n}\left(  1/k\right)  .$ Since the first positive eigenvalue of
$\mathbb{S}^{n}\left(  1/k\right)  $ is equal to $nk^{2}$, from Poincar\'{e}
inequality we obtain
\begin{align*}
F(V)  &  \geq-k^{2}+nk^{2}\sum_{l=1}^{n+1}\int_{\mathbb{S}^{n}\left(
1/k\right)  }\left\vert a_{l}\right\vert ^{2}\\
&  =-k^{2}+nk^{2}\int_{\mathbb{S}^{n}\left(  1/k\right)  }\left\vert
V\right\vert ^{2}=\left(  n-1\right)  k^{2}.
\end{align*}

\end{proof}

\begin{lemma}
\label{trans}Assume that a Lie subgroup $H$ of $O(n)=\operatorname*{Iso}%
\left(  \mathbb{S}^{n}\left(  1/k\right)  \right)  $ acts transitively on
$\mathbb{S}^{n}\left(  1/k\right)  .$ Then%
\begin{equation}
\int_{H}\left\langle hu,v\right\rangle dh=0 \label{h}%
\end{equation}
for any fixed vectors $u,v\in\mathbb{S}^{n}\left(  1/k\right)  .$
\end{lemma}

\begin{proof}
Let $u,v\in\mathbb{S}^{n}\left(  1/k\right)  $ be given. There is $g\in H$
such that $g(v)=-v.$ Since the left translation $L_{g^{-1}}:H\rightarrow H,$
$L_{g^{-1}}(h)=g^{-1}h$ is an orientation preserving isometry of $H$ we have%
\[
\int_{H}\left\langle hu,v\right\rangle dh=\int_{H}\left\langle g^{-1}%
hu,v\right\rangle \left(  L_{g^{-1}}\right)  ^{\ast}dh=\int_{H}\left\langle
hu,gv\right\rangle dh=-\int_{H}\left\langle hu,v\right\rangle dh
\]
which proves (\ref{h}).
\end{proof}

\bigskip

The next lemma will also be used to prove Theorem \ref{ricci}. We make use of
the following terminology: An orbit of highest dimension of a compact Lie
group $G$ acting on a compact manifold $M$ is called a principal orbit (except
to the exceptional orbits, see \cite{HL}). We say that $G$ acts with
cohomogeneity one if the principal orbits of $G$ have codimension $1.$ The
principal orbits of $G$ foliates a open dense subset of $M$ whose
complementary has zero $\dim M-$dimensional measure (\cite{HL}).

\begin{lemma}
\label{ortog}Let $M^{n}$ be a compact, orientable Riemannian manifold. Let $G$
be a compact Lie subgroup of $\operatorname*{Iso}(M)$ acting with
cohomogeneity one on $M$. Assume moreover that the subgroup of isotropy of $G$
at any point of a principal orbit of $G$ acts transitively (by the derivative)
on the spheres centered at origin of the tangent space of the orbit at the
point$.$ Then any $G-$invariant vector field is orthogonal to the principal
orbits of $G.$
\end{lemma}

\begin{proof}
Let $p\in M$ be such that $G(p)$ is a principal orbit of $G$ and let $v\in
T_{p}G(p)$ be any fixed vector$.$ Since $X$ is $G-$invariant we have
\[
\left\langle X(p),v\right\rangle =\left\langle dg_{p}^{-1}%
(X(g(p))),v\right\rangle =\left\langle X(g(p)),dg_{p}v\right\rangle
\]
for all $g\in G,$ so that
\[
\left\langle X(p),v\right\rangle =\frac{1}{\operatorname*{Vol}(G)}\int
_{G}\left\langle X(g(p)),dg_{p}v\right\rangle dg.
\]
Denoting by $H$ be the isotropy subgroup of $G$ at $p$ we have by coarea
formula%
\[
\int_{G}\left\langle X(g(p)),dg_{p}v\right\rangle dg=\int_{G/H}\left(
\int_{gH}\left\langle X((gh)(p)),d(gh)_{p}v\right\rangle d(gh)\right)  d(gH).
\]
Moreover%
\begin{align*}
\int_{gH}\left\langle X((gh)(p)),d(gh)_{p}v\right\rangle d(gh)  &  =\int
_{H}\left\langle X(g(p)),d\left(  gh\right)  _{p}v\right\rangle dh\\
&  =\int_{H}\left\langle dg_{p}^{-1}X(g(p)),dh_{p}v\right\rangle dh=0
\end{align*}
by the previous lemma, since the action $h\mapsto dh_{p}$ of $H$ on
$T_{p}G(p)$ is transitive on the spheres of $T_{p}G(p)$. Then
\[
\left\langle X(p),v\right\rangle =\frac{1}{\operatorname*{Vol}(G)}\int
_{G}\left\langle X(g(p)),dg_{p}v\right\rangle dg=0.
\]
This implies that $X\left(  p\right)  \in\left(  T_{p}G\left(  p\right)
\right)  ^{\perp}$ since $v$ is arbitrary, concluding with the proof of the lemma.
\end{proof}

\begin{proof}
[Proof of Theorem \ref{inf}]A calculation shows that if $X\in C^{\infty
}(T\mathbb{S}^{n}\left(  1/k\right)  )$ is the orthogonal projection on
$T\mathbb{S}^{n}\left(  1/k\right)  $ of a vector $v\in\mathbb{R}^{n+1}$ then
$-\operatorname*{div}\nabla X=k^{2}X.$ We will prove that $k^{2}$ is the
infimum of $F$ and hence proving the theorem. We may assume that
$v\in\mathbb{S}^{n}\left(  1/k\right)  .$

Let $G$ be the isotropy subgroup of $\operatorname*{Iso}(\mathbb{S}^{n}\left(
1/k\right)  )$ at $v$ and $W\in C^{\infty}(T\mathbb{S}^{n}\left(  1/k\right)
)$ assuming the infimum of $F.$ It follows from Lemma \ref{laplafunc} that if
$W$ has zero $G-$mean, then $F(W)\geq\left(  n-1\right)  k^{2}$ and the
theorem is proved in this case. We may then assume that $W_{G}$ is non zero$.$
Since $\operatorname{div}W_{G}=\left(  \operatorname{div}W\right)  _{G},$ $F$
assumes its infimum at $W_{G}$ too$.$

By the Lemma \ref{ortog}, $W_{G}$ is orthogonal to the orbits of $G,$ which
are geodesic spheres centered at $v.$ We then have $W_{G}=\left\langle
W_{G},\operatorname{grad}s\right\rangle \operatorname{grad}s,$ where $s$ is
the distance in $\mathbb{S}^{n}\left(  1/k\right)  $ to $v.$ Define $h\in
C^{2}\left(  \left[  0,\pi/k\right]  \right)  $ by $h(t)=\left\langle
W_{G},\operatorname{grad}s\right\rangle (x),$ where $x\in\mathbb{S}^{n}\left(
1/k\right)  $ is such that $t=s(x).$ Since $W_{G}$ is $G-$invariant $h$ is
well defined and we have $W_{G}=h(s)\operatorname{grad}s.$

If $\phi\in C^{2}\left(  \left[  0,\pi/k\right]  \right)  $ is a primitive of
$h$ and $f:\mathbb{S}^{n}\left(  1/k\right)  \rightarrow\mathbb{R}$ is defined
by $f(x)=\phi(s(x))$ then we have
\[
\operatorname{grad}f=\phi^{\prime}\operatorname{grad}s=h\operatorname{grad}%
s=W_{G}.
\]

Applying Reilly's formula to $f$ (see \cite{Re}) we obtain%
\begin{align*}
\int_{\mathbb{S}^{n}\left(  1/k\right)  }\left(  \Delta f\right)  ^{2}dx  &
=\int_{\mathbb{S}^{n}\left(  1/k\right)  }\operatorname*{Ric}%
\nolimits_{\mathbb{S}^{n}\left(  1/k\right)  }\left(  \operatorname{grad}%
f,\operatorname{grad}f\right)  dx\\
&  +\int_{\mathbb{S}^{n}\left(  1/k\right)  }\left\vert \operatorname*{Hess}%
\left(  f\right)  \right\vert ^{2}dx\\
&  =\int_{\mathbb{S}^{n}\left(  1/k\right)  }\operatorname*{Ric}%
\nolimits_{\mathbb{S}^{n}\left(  1/k\right)  }\left(  W_{G},W_{G}\right)
dx+\int_{\mathbb{S}^{n}\left(  1/k\right)  }\left\vert \operatorname*{Hess}%
\left(  f\right)  \right\vert ^{2}dx,
\end{align*}
where $\Delta f$ and $\left\vert \operatorname*{Hess}\left(  f\right)
\right\vert $ are the Laplacian and the norm of the Hessian of $f.$ Note that
$\left\vert \operatorname*{Hess}\left(  f\right)  \right\vert =\left\vert
\nabla W_{G}\right\vert $ and then, since $\left(  \Delta f\right)  ^{2}\leq
n\left\vert \operatorname*{Hess}\left(  f\right)  \right\vert ^{2},$ assuming
that $\left\vert W_{G}\right\vert _{L^{2}}=1,$ we obtain%
\begin{align*}
&  \left(  n-1\right)  \int_{\mathbb{S}^{n}\left(  1/k\right)  }\left\vert
\nabla W_{G}\right\vert ^{2}dx\\
&  \geq\int_{\mathbb{S}^{n}\left(  1/k\right)  }\left\vert W_{G}\right\vert
^{2}\operatorname*{Ric}\nolimits_{\mathbb{S}^{n}\left(  1/k\right)  }\left(
\frac{W_{G}}{\left\vert W_{G}\right\vert },\frac{W_{G}}{\left\vert
W_{G}\right\vert }\right)  dx\geq(n-1)k^{2}.
\end{align*}
It follows that $F(W_{G})\geq k^{2},\ $proving the first part of the theorem.
For the last part, it follows from Bochner-Yano formula of Remark \ref{by}
that if $X$ is a Killing field in $\mathbb{S}^{n}\left(  1/k\right)  $ then
$F(X)=(n-1)k^{2}$ so that $F(X)=k^{2}$ if and only if $n=1.$
\end{proof}

\section{The critical points of the energy on a rank 1 compact symmetric
space}

\qquad In this section we prove:

\begin{theorem}
\label{main}Let $M$ be a rank $1$ compact symmetric space. Let $G$ be a
compact Lie subgroup of $\operatorname*{Iso}(M)$ that leaves pointwise fixed a
totally geodesic submanifold of $M$ with dimension bigger than or equal to
$1.$ Then the all the critical points of the energy in $M$ are assumed by a
$G-$invariant vector field.
\end{theorem}

If $V$ is a vector field on $M$ and if $h:M\rightarrow M$ is a diffeomorphism,
we denote by $V^{h}$ the $h-$related vector field to $V,$ that is,
$V^{h}(p)=\left(  dh_{p}\right)  ^{-1}(V(h(p)),$ $p\in M.$ We use the
following lemma.

\begin{lemma}
\label{lim}Let $p\in M$, $v\in T_{p}M$ and $V\in C^{\infty}(TM)$ be given.
Assume that $g(p)=p$ and $dg_{p}(v)=v$ for all $g\in G$. Given $q\in M$ assume
that $h\in\operatorname*{Iso}\left(  M\right)  $ is such that $h(p)=q$ and
$dh_{q}^{-1}(V(q))=v$. Let $x_{n}\in M$ be a sequence converging to $p$. Then%
\[
V^{h}(p)=\lim_{n\rightarrow\infty}\frac{\int_{G}dg_{x_{n}}^{-1}(V^{h}%
(g(x_{n})))}{\operatorname*{Vol}\left(  G(x_{n})\right)  },
\]
where $\operatorname*{Vol}\left(  G(x_{n})\right)  $ is the $k-$dimensional
Hausdorff measure of $G(x_{n}),$ $k=\dim G(x_{n})\geq0.$
\end{lemma}

\begin{proof}
We will prove that
\[
\left\langle V^{h}(p),Z(p)\right\rangle =\left\langle \lim_{n\rightarrow
\infty}\frac{\int_{G}dg_{x_{n}}^{-1}(V^{h}(g(x_{n})))}{\operatorname*{Vol}%
(G(x_{n}))},Z(p)\right\rangle \text{, }Z\in C^{\infty}(TM).
\]
Given $Z\in C^{\infty}(TM)$ we have, for a given $n,$%
\[
\left\langle \frac{\int_{G}dg_{x_{n}}^{-1}(V^{h}(g(x_{n})))}%
{\operatorname*{Vol}(G(x_{n}))},Z(x_{n})\right\rangle =\frac{\int
_{G}\left\langle dg_{x_{n}}^{-1}(V^{h}(g(x_{n}))),Z(x_{n})\right\rangle
}{\operatorname*{Vol}(G(x_{n}))}%
\]
so that
\begin{align*}
\underset{g\in G}{\inf}\left\langle dg_{x_{n}}^{-1}(V^{h}(g(x_{n}%
))),Z(x_{n})\right\rangle  &  \leq\frac{\int_{G}\left\langle dg_{x_{n}}%
^{-1}(V^{h}(g(x_{n}))),Z(x_{n})\right\rangle \omega}{\operatorname*{Vol}%
(G(x_{n}))}\\
&  \leq\underset{g\in G}{\sup}\left\langle dg_{x_{n}}^{-1}(V^{h}%
(g(x_{n}))),Z(x_{n})\right\rangle .
\end{align*}

Letting $n\rightarrow\infty$\ we obtain
\begin{align*}
\underset{g\in G}{\inf}\left\langle dg_{p}^{-1}(V^{h}%
(g(p))),Z(p)\right\rangle  &  \leq\left\langle \lim_{n\rightarrow\infty}%
\frac{\int_{G}dg_{x_{n}}^{-1}(V^{h}(g(x_{n}))dg}{\operatorname*{Vol}%
(G(x_{n}))},Z(p)\right\rangle \\
&  \leq\underset{g\in G}{\sup}\left\langle dg_{p}^{-1}(V^{h}%
(g(p))),Z(p)\right\rangle ,
\end{align*}
and, since $g(p)=p,$%
\begin{align*}
\underset{g\in G}{\inf}\left\langle dg_{p}^{-1}(V^{h}(p)),Z(p)\right\rangle
&  \leq\left\langle \lim_{n\rightarrow\infty}\frac{\int_{G}dg_{x_{n}}%
^{-1}(V^{h}(g(x_{n})))dg}{\operatorname*{Vol}(G(x_{n}))},Z(p)\right\rangle \\
&  \leq\underset{g\in G}{\sup}\left\langle dg_{p}^{-1}(V^{h}%
(p)),Z(p)\right\rangle .
\end{align*}
Since
\[
dg_{p}^{-1}(V^{h}(p))=dg_{p}^{-1}(v)=v=V^{h}(p)
\]
for all $g\in G$ it follows that
\[
\left\langle V^{h}(p),Z(p)\right\rangle \leq\left\langle \lim_{n\rightarrow
\infty}\frac{\int_{G}dg_{x_{n}}^{-1}(V^{h}(g(x_{n})))dg}{\operatorname*{Vol}%
(G(x_{n}))},Z(p)\right\rangle \leq\left\langle V^{h}(p),Z(p)\right\rangle
\]
proving the lemma.
\end{proof}

\bigskip

\begin{proof}
[Proof of Theorem \ref{main}]Let $V\in C^{\infty}(TM)$ be a critical point of
$F$ with unit $L^{2}-$norm. Since $V$ is non zero, there is $q\in M$ such that
$V\left(  q\right)  \neq0.$ Let $G$ be a Lie subgroup of $\operatorname*{Iso}%
(M)$ that leaves pointwise fixed a totally geodesic submanifold $T$ of $M,$
$\dim T\geq1.$ Choose a $p\in T$ and a nonzero vector $v\in T_{p}T.$ We have
$g(p)=p$ for all $g\in G$ and we claim that $dg_{p}(v)=v$ for all $g\in G$
too. Up to a multiplication of $v$ by a positive real number, we may assume
that $v$ is contained on a normal geodesic ball of $T_{p}T$. Set
$w=dg_{p}(v).$ Then, since $\exp_{p}v\in T,$ where $\exp_{p}$ it the
Riemannian exponential, we have
\[
\exp_{p}v=g(\exp_{p}v)=\exp_{g(p)}dg_{p}(v)=\exp_{p}w
\]
and then $v=w$.

Since $M$ is a compact rank $1$ symmetric space, there is an isometry $h$ of
the $M$ such that $h\left(  p\right)  =q$ and $dh_{p}^{-1}\left(  V\left(
q\right)  \right)  =v.$ Let $V^{h}$ be the vector field $h-$related to $V.$ We
claim that $V_{G}^{h}$ is not identically zero. Indeed, taking a sequence
$x_{n}\in M$ converging to $p$, we have, by the Lemma \ref{lim}
\[
V^{h}(p)=\lim_{n\rightarrow\infty}\frac{\int_{G}dg_{x_{n}}^{-1}(V^{h}%
(g(x_{n})))}{\operatorname*{Vol}(G(x_{n}))}=\underset{n\rightarrow\infty}%
{\lim}\frac{V_{G}^{h}\left(  x_{n}\right)  }{\operatorname*{Vol}(G(x_{n}))}.
\]
Hence, if $V_{G}^{h}\equiv0$ then $V^{h}\left(  p\right)  =0,$ a
contradiction! This proves the theorem.
\end{proof}

\section{\label{Rig}A characterization of spheres}

\qquad Rotationally symmetric manifolds are well known and much used as models
on comparison theorems on Geometric Analysis. We consider here a
generalization of such manifolds which we call \emph{two point symmetric
manifolds}:

\begin{definition}
\label{def}We say that a Riemannian $n-$dimensional manifold $M,$ $n\geq2,$ is
two point symmetric with center $p\in M$ if the isotropy subgroup
$G:=\operatorname*{Iso}\nolimits_{p}(M)$ of the isometry group of $M$ at $p$
is isomorphic to the isotropy subgroup of the isometry group of a two point
homogeneous space $S$ at any point in $S.$
\end{definition}

We note that the above definition is a natural generalization of the way used
in \cite{Cho} to define rotationally symmetric spaces. We also observe that
this definition is equivalent to the requirement that given points
$p_{1},p_{2},p_{3},p_{4}\in M$ that belong to a common geodesic sphere of $M$
centered at $p,$ if $d(p_{1},p_{2})=d(p_{3},p_{4}),$ where $d$ is the
Riemannian distance in $M$, then there is an isometry $i\in G$ such that
$i(p_{1})=p_{3}$ and $i(p_{2})=p_{4}$ (see \cite{Cha}, \cite{He})$.$ It is
also known that this two point homogenous characterization of $M$ around $p$
is equivalent to the isotropy of $G$ at any point $x\in M$ of a principal
orbit $G(x)$ of $G$ acting transitively on the Euclidean spheres centered at
the origin of $T_{x}G(x)$ (\cite{He})$.$

We also recall that a two point homogeneous space is isometric to a rank $1$
symmetric space. Accordingly to the symmetric space classification, it follows
that a Riemannian manifold $M$ is two point homogenous with center $p\in M$ if
$\operatorname*{Iso}\nolimits_{p}(M)$ is isomorphic to one of the following
Lie groups: $O(n),$ $U(1)\times U(n),$ $Sp(1)\times Sp(n)$ or
$\operatorname*{Spin}(9)$ (see \cite{He})$.$ When $\operatorname*{Iso}%
\nolimits_{p}(M)=O(n)$ the space $M$ is rotationally symmetric.

\begin{theorem}
\label{ricci}Let $M^{n}$ be a compact two point symmetric space with center
$p$, $n\geq2,$ with positive Ricci curvature $\operatorname*{Ric}%
\nolimits_{M}$ and assume that $\operatorname*{Ric}\nolimits_{M}\geq
(n-1)k^{2},$ $k>0$. Set $G=\operatorname*{Iso}\nolimits_{p}(M)$. Then the
infimum of $F$ on the\ $G-$invariant vector fields is bigger than or equal to
$k^{2}$ and the equality holds if and only if $M$ is a sphere of radius $1/k$.
\end{theorem}

\begin{proof}
Denote by $s:M\rightarrow\mathbb{R}$ the distance in $M$ to $p.$ Note that the
level sets of $s$ are geodesic spheres centered at $p\ $and that the mean
curvature and the norm of the second fundamental form of these geodesic
spheres depend only on $s.$ Set $l=\max s.$

Let $V\in C^{\infty}(TM)$ a $G-$invariant vector field such that $F(V)$ is the
positive infimum of $F$ on the space of $G-$invariant vector fields. Since the
subgroup of isotropy of $G$ at a point $p$ of a principal orbit of $G$ acts
transitively (by the derivative) on the spheres centered at origin of
$T_{p}G(p)$\ it follows from Lemma \ref{ortog} that $V$ may be written on the
form $V=\left\langle V,\operatorname{grad}s\right\rangle \operatorname{grad}.$
Define $h,f$ and $\phi$ as in Theorem \ref{inf}. The same reasoning used in
this theorem allows us to conclude that $F(V)\geq k^{2},\ $proving the first
part of the theorem. By Theorem \ref{inf} we know that if $M$ is a sphere of
radius $1/k$ then $F(V)=k^{2}.$ We now prove the converse. Thus, assume that
$F(V)=k^{2}.$ From the inequality obtained from Reilly's formula in Theorem
\ref{inf} we obtain
\begin{equation}
\left(  \Delta f\right)  ^{2}=n\left\vert \operatorname*{Hess}\left(
f\right)  \right\vert ^{2}. \label{laplacian}%
\end{equation}

Given $s\in\left(  0,l\right)  $ denote by $\left\vert B\right\vert \left(
s\right)  $ and $H\left(  s\right)  $ the norm of the second fundamental form
and the mean curvature of the geodesic sphere at a distance $s$ of $p$, with
respect to the unit normal vector pointing to the center $p.$ Since
$f=\phi\circ s,$ straightforward calculations give%
\[
\left(  \Delta f\right)  ^{2}=\left(  \phi^{\prime\prime}\right)
^{2}-2\left(  n-1\right)  \phi^{\prime\prime}\phi^{\prime}H+\left(
n-1\right)  ^{2}\left(  \phi^{\prime}\right)  ^{2}H^{2}%
\]
and%
\[
\left\vert \operatorname*{Hess}\left(  f\right)  \right\vert ^{2}=\left(
\phi^{\prime\prime}\right)  ^{2}+\left(  \phi^{\prime}\right)  ^{2}\left\vert
B\right\vert ^{2}.
\]
Then (\ref{laplacian}) is equivalent to%
\begin{equation}
\left(  \phi^{\prime\prime}+H\phi^{\prime}\right)  ^{2}=\left(  \phi^{\prime
}\right)  ^{2}\left[  -\frac{n\left\vert B\right\vert ^{2}}{n-1}%
+nH^{2}\right]  \label{f}%
\end{equation}
which implies that%
\begin{equation}
-\frac{\left\vert B\right\vert ^{2}}{n-1}+H^{2}\geq0 \label{B}%
\end{equation}
since $\phi^{\prime}$ cannot be identically zero (otherwise $V$ would be$,$
which is not the case). But, denoting by $\lambda_{i}$ the principal
curvatures of the geodesic spheres, we have
\begin{align*}
-\frac{\left\vert B\right\vert ^{2}}{n-1}+H^{2}  &  =\frac{-\left(
\overset{n-1}{\underset{i=1}{%
{\displaystyle\sum}
}}\lambda_{i}^{2}\right)  }{n-1}+\frac{\left(  \underset{j=1}{\overset{n-1}{%
{\displaystyle\sum}
}}\lambda_{j}\right)  ^{2}}{\left(  n-1\right)  ^{2}}\\
&  =-\sum_{i,j=1,\text{ }i<j}^{n}\left(  \lambda_{i}-\lambda_{j}\right)
^{2}\leq0
\end{align*}
so that
\[
-\frac{\left\vert B\right\vert ^{2}}{n-1}+H^{2}=0,
\]
and, from (\ref{f}),
\begin{equation}
\phi^{\prime\prime}+H\phi^{\prime}=0. \label{phi}%
\end{equation}

Since
\[
V=\operatorname{grad}f=\left(  h\circ s\right)  \operatorname{grad}s
\]
and $V$ is an infimum of $F$ we have
\[
-\operatorname{div}\nabla\left(  h(s)\operatorname{grad}s\right)
=k^{2}h(s)\operatorname{grad}s.
\]

A calculation gives%
\[
\operatorname{div}\nabla\left(  h\operatorname{grad}s\right)  =\left(
h^{\prime\prime}(s)-\left(  n-1\right)  H\left(  s\right)  h^{\prime}\left(
s\right)  -\left\vert B\right\vert ^{2}\left(  s\right)  h\left(  s\right)
\right)  \operatorname{grad}s
\]
so that%
\begin{equation}
h^{\prime\prime}-\left(  n-1\right)  Hh^{\prime}-\left\vert B\right\vert
^{2}h=-k^{2}h. \label{p}%
\end{equation}

But from (\ref{phi}) we have%
\[
h^{\prime\prime}-\left(  n-1\right)  Hh^{\prime}-\left\vert B\right\vert
^{2}h=h^{\prime\prime}+(n-1)\left[  H^{2}-\left\vert B\right\vert ^{2}/\left(
n-1\right)  \right]  h=h^{\prime\prime}%
\]
and hence%
\[
h^{\prime\prime}\left(  s\right)  =-k^{2}h(s).
\]

It follows that $h$ is of the form
\[
h(s)=A\cos(sk)+B\sin\left(  ks\right)
\]
and, since $\phi^{\prime}=h,$%
\[
\phi(s)=\frac{A}{k}\cos(sk)+\frac{B}{k}\sin\left(  ks\right)  .
\]

From $f(x)=\phi(s(x))$ we then obtain using (\ref{phi})%
\[
\Delta f=\phi^{\prime\prime}+\phi^{\prime}\Delta s=\phi^{\prime\prime
}-(n-1)H\phi^{\prime}=\phi^{\prime\prime}+\left(  n-1\right)  \phi
^{\prime\prime}=n\phi^{\prime\prime}=-nk^{2}\phi=-nk^{2}f.
\]
Hence $nk^{2}$ is a eigenvalue of the usual Laplacian. Under the hypothesis
$\operatorname*{Ric}\nolimits_{M}\geq\left(  n-1\right)  k^{2}$ this implies
that $M$ is a sphere of radius $1/k$ (\cite{O})$,$ finishing the proof of the theorem.
\end{proof}

\end{document}